\documentclass[12pt,usenames,dvipsnames]{article}
\usepackage[a4paper]{anysize}\marginsize{2cm}{2cm}{2cm}{2cm}
\pdfpagewidth=\paperwidth \pdfpageheight=\paperheight
\usepackage{amsfonts,amssymb,amsthm,amsmath,eucal,tabu,hyperref}
\usepackage{xcolor}
\usepackage{float}
\usepackage{abstract}

\usepackage{wrapfig}
\usepackage{pgfplots}
\pgfplotsset{compat=1.15}
\usepackage{mathrsfs}
\usetikzlibrary{arrows}
\usepackage{diagbox}
\usepackage{subcaption}

\theoremstyle{plain}
\newtheorem{thm}{Theorem}[section]
\newtheorem{theorem}[thm]{Theorem}

\newtheorem{corollary}[thm]{Corollary}

\theoremstyle{definition}
\newtheorem{definition}[thm]{Definition}
\newtheorem{remark}[thm]{Remark}
\newtheorem{notation}[thm]{Notation}
\newtheorem{example}[thm]{Example}
\newtheorem{claim}[thm]{Claim}
\newtheorem{question}[thm]{Question}

\DeclareMathOperator{\rank}{rank}

\def\PP{\mathbb{P}}
\def\vdim{\text{vdim}}

\def\K{\mathbb{K}}
\def\a0n{a_0,\ldots,a_N}
\def\x0n{x_0,\ldots,x_N}
\def\b0n{b_0,\ldots,b_N}
\def\y0n{y_0,\ldots,y_N}
\def\p0n{p_0,\ldots,p_N}

\usepackage{xargs} 
\usepackage{xcolor} 
\usepackage[colorinlistoftodos,prependcaption,textsize=tiny]{todonotes}
\newcommandx{\unsure}[2][1=]{\todo[linecolor=red,backgroundcolor=red!25,bordercolor=red,#1]{#2}}
\newcommandx{\change}[2][1=]{\todo[linecolor=Orange,backgroundcolor=orange!25,bordercolor=orange,#1]{#2}}
\newcommandx{\info}[2][1=]{\todo[linecolor=SkyBlue,backgroundcolor=Yellow!35,bordercolor=SkyBlue,#1]{#2}}
\newcommandx{\improvement}[2][1=]{\todo[linecolor=Plum,backgroundcolor=Plum!25,bordercolor=Plum,#1]{#2}}

\title{XXX}
\author{Marcin Dumnicki, Grzegorz Malara, Halszka Tutaj-Gasi\'nska}
\date{\today}

\begin{document}
\title{Matrixwise (approach to unexpected hypersurfaces) Reloaded}
\maketitle

\begin{abstract}
    In the paper we provide a new method of proving the existence of a hypersurface of degree $d$ in $\PP^n$, with a general point of multiplicity $m$ and vanishing at a given set of points $Z$,  by looking at weak combinatorics of a set $Z$. This method has a direct application in the theory of unexpected hypersurfaces, where many of the examples are based only on computer experiments.

    \bigskip
    \noindent \textbf{Keywords:} unexpected hypersurfaces, interpolation matrix, determinant

    \noindent \textbf{MSC classification:} 14N20, 14N05, 14M15, 15A06
\end{abstract}

\section{Introduction}
 The notion of an unexpected hypersurface was introduced in \cite{CHMN16}, where the authors,
 using the results of \cite{DIV}, gave an example of a quartic curve in $\PP^2$
that passes through a certain special set $Z$ of nine points imposing independent conditions on quartics; and which has a triple point in a general point $B$. 
Counting the condition for a quartic to pass through $Z$ and to have a triple (general) point, we get that we expect exactly 0 such quartic. However, it exists, so we call it an unexpected curve, and it is said to be of type $(4,3)$.  After \cite{CHMN16} appeared, the
notion of unexpected curve was generalized, also to unexpected hypersurfaces; see e.g., the introduction to \cite{MT-G}. Namely, take $\K[x_0,\ldots,x_n]$, the homogeneous coordinate ring of $\PP^n=\PP^n(\K)$, where $\K$ is any field of characteristic zero.  Let
$Z\subseteq \PP^n$ be a scheme with its saturated homogeneous ideal $I_Z$. Now, let $B_i\subset \PP^n$, $1\leq i\leq r$, be general linear varieties disjoint from $Z$.
For integers $m_i >0$, take the scheme $B=m_1B_1+\cdots+m_r B_r$ with the
ideal $I_B$. Let $ H_B$ be the Hilbert polynomial of $\K[x_0,\ldots,x_n]/I_B$.
\begin{definition}[\cite{HMNT}]
We say that a hypersurface $H$ of degree $d$ passing through $Z$  and passing through each $B_i$ with multiplicity $m_i$ is unexpected if 
$$\dim [I_{X\cup Z}]_d > \max(0, \dim [I_Z]_d - H_B(d));$$
i.e., $H$ is unexpected if vanishing on $B$ imposes on $[I_Z]_d$ fewer than the expected
number of conditions. We also assume that the conditions imposed by vanishing in $Z$ on the forms of degree $d$ are independent.
\end{definition}

Since \cite{CHMN16}, more papers on unexpected varieties have appeared, e.g., \cite{BMSS,DHRST18,HMNT}, and \cite{DFHMST} in particular, presenting more examples of unexpected hypersurfaces.

In the most of the cases only one general linear subvariety of dimension 0 (a point) is considered, leading to the situation where the set of points $Z$ imposes independent conditions on hypersurfaces of degree $d$, but imposing additionally vanishing at a general point $B$ with multiplicity $m$ gives dependency. In the most interesting cases, the ``unexpected" hypersurface is irreducible, and its existence is not obvious. So far, apart from several cases, the existence of unexpected curves in $\PP^2$ has been proven by using the theory of logarithmic bundles (or syzygy bundles) and their splitting type, or simply by calculating the equation of a hypersurface in a computer algebra package.

Our aim is to provide a new, geometrically-combinatorial method of proving the existence of such hypersurfaces; or rather, of proving the existence of hypersurfaces of a given degree and a given multiplicity in a general point, not caring, at first, whether the hypersurface is unexpected or not. The appearance of the quartic from \cite{CHMN16,DIV}, mentioned at the beginning, follows (by our approach) from the fact that the set $Z$ of nine points has the following property: there exist three lines passing through 4 points of $Z$ and four lines passing through 3 points of $Z$ (see Example \ref{ex:A4k+1} for details). To be more precise: by our method, every set of 9 points with this combinatorics (over any field) would admit a quartic vanishing at these points and additionally with triple point in general position. (Note that by \cite{FGST} all such are projectively equivalent).

The situation is very similar in many cases in $\PP^2$ and some cases in $\PP^3$. Namely, we investigate the ``weak combinatorics'' 
of the set of points $Z$. This means we are interested only in the degree of the arrangement, the number of points on each hyperplane of it, then the number of points on the intersection of the hyperplanes etc., as well as the number of lines (planes, spaces, etc.) through each point of $Z$, cf. e.g., \cite{Pokora}.
This weak combinatorics allows us to prove the existence of a curve (or hypersurfaces) with desired properties. Another, rather astonishing, example is in $\PP^3$: whenever we have a set $Z$ of 12 points, lying on 12 planes, with each plane containing 6 points and each point belonging to 6 planes (the $D4$ configuration has these properties), then the unexpected hypersurface of degree 3 with a general point of multiplicity 3 exists.

We think that this combinatorial nature of unexpectendess is interesting and it is worth studying. 
There may be a connection with another method of showing the existence of unexpected curves; namely by computing the splitting type, which is possible to be done "by hand" in case of, for example, supersolvable or inductively free arrangements, \cite{D,DMO,Janasz}. We  we will try to understand this connection in further research. 

We found a series of examples where our method does not work (up to our current understanding), but we do not know whether it \emph{cannot} work or just we miss some combinatorics that we should use.

Concerning the methods, the present paper is in a sense a continuation of \cite{DFHMST}. 
In that paper, the authors considered a set $Z$ of points in $\PP^n$ (where $\PP^n$ has coordinates $(x_0,\ldots,x_n)$) such that $Z$ and a general point
$B=(a_0,\ldots,a_n)$ with multiplicity $m$ impose independent conditions on forms of degree $d$,
with the affine dimension of the system
of forms of degree $d$ vanishing on $Z$ and on
$mB$ being 1. Then they proved that
the equation of the  hypersurface of degree $d$ and passing through $Z$ and vanishing in $B$ to order $m$,  is a determinant $D$ of a suitable interpolation matrix $M$,  
involving two groups of coordinates, $(a_0,\ldots,a_n)$, $(x_0,\ldots,x_n)$. This $D$, if not
 identically zero, may be seen as a bihomogeneous polynomial \big(of bidegree $\left(\binom{m+n-1}{n} (d-m+1),d\right)$
and so defines a variety in $\PP^n\times\PP^n$. Studying this variety, the authors were able, among others, to explain the phenomenon of the so-called BMSS duality, see \cite{BMSS,DFHMST}.

In the present paper, we change a point of view. We consider an interpolation matrix $M$ for forms of degree $d$, vanishing in a given set of points $Z$ 
(with $|Z|=s$)
and vanishing at a point $B$ with multiplicity $m$,
where $d$, $m$, and $s$ are such that $M$ is a square matrix. 
Let $F$ denote the determinant of this matrix, seen as a function of the coordinates of $B$. Then the set $\{B: F(B)=0\}$ describes a locus of such points $B$, that there exists a form of degree $d$ and vanishing at a given set $Z$, as well as vanishing at a point $B$ with multiplicity $m$, i.e., the linear system ${\cal L}(d; mB+Z)$ is non-empty.
Thus, to obtain a hypersurface vanishing at $Z$ and vanishing to order $m$ in a \textit{general} $B$, we have to prove that $F\equiv 0,$ so the locus is the whole $\PP^n$.

The paper is organized as follows. 
In Section \ref{sec:mt} we prove 
our main theorem, which enables to bound from below a multiplicity of a point in the set $F=0$. 
In next sections we show how we may explain the existence of unexpected curves/hypersurfaces by obtaining, 
from Theorem \ref{thm:main}, that $F\equiv 0$. We begin with the simples cases in $\PP^2$, where $|Z|+\binom{m+1}{2}=\binom{d+2}{2}$ (see Section \ref{sec:ex1}). Then we pass to the situation where $|Z|+\binom{m+1}{2}=\binom{d+2}{2}+1$ (Section \ref{sec:ex2}), where we are also able to provide a nice combinatorial condition. In this Section we also briefly describe the limits of our method, by applying it to the list of known arrangements. In Section \ref{sec:ex3} we provide an interesting example, where the difference between $d$ and $m$ is greater than one, and to prove the existence we need also some curves of degree 4 passing through a subset of the points of $Z$. The last Section contains two examples in $\PP^3$. We focus on presenting examples because of two reasons: we think that that way it is easier to understand the method and to use it for somebody's own research; furthermore, there is a broad list of sporadic examples of unexpected hypersurfaces, so we must investigate it case by case.

\section{Main theorem}\label{sec:mt}
Denote by $Z \subset \PP^N$ a set of $s$  distinct points $\{P_1, \ldots ,P_s\}$ with coordinates $(p_{i0},\ldots,p_{iN})$, for $i=1,\ldots,s$. Let $d$ and $m$ be positive integers such that the following equality is satisfied:
\begin{equation}\label{dim}
 \binom{d+N}{N}=\binom{m+N-1}{N}+s.
 \end{equation}
Fix the vector $w$ consisting of all monomials of degree $d$ in variables $(a_0, \dots, a_N)$, arranged in a chosen order, for example
$$w(\a0n):=(a_{0}^d, a_0^{d-1}a_1, a_0^{d-1}a_2,\ldots,a_{N-1}a_N^{d-1},a_N^{d}).$$ 
Define a matrix $M$ as follows. The first
$s$ rows correspond to $w(P_1), \ldots, w(P_s)$; the next
$\binom{m+N-1}{N}$ rows represent all the partial derivatives of $w$ of order $m-1$.
(As before, the order of differentiation does not matter, but we fix one for consistency.)
 Thus the matrix $M$ is a square matrix of the form:
\begin{equation}
\label{eq:defM}
M=\left[\begin{array}{c}
w(P_1)\\
\vdots\\
w(P_s)\\
\frac{\partial^{m-1}w}{\partial a_0^{m-1} } \\
\frac{\partial^{m-1}w}{\partial a_0^{m-2}\partial a_1 } \\
\vdots\\
\frac{\partial^{m-1}w}{\partial a_N^{m-1} } 
\end{array}\right]=
\end{equation}
{\tiny	$$
	\left[\begin{array}{cccccc}
    	p_{10}^d& p_{10}^{d-1}p_{11} & p_{10}^{d-1}p_{12}&\ldots&p_{1N-1}p_{1N}^{d-1}&p_{1N}^{d}\\
	\vdots&\vdots&\vdots& &\vdots&\vdots\\
	p_{s0}^d& p_{s0}^{d-1}p_{s1} & p_{s0}^{d-1}p_{s2}&\ldots&p_{sN-1}p_{sN}^{d-1}&p_{sN}^{d}\\
	\frac{\partial^{m-1}}{\partial a_0^{m-1} }(a_0^d)&	\frac{\partial^{m-1}}{\partial a_0^{m-1} }(a_0^{d-1}a_1) &	\frac{\partial^{m-1}}{\partial a_0^{m-1} }(a_0^{d-1}a_2) &
	\dots&
	\frac{\partial^{m-1}}{\partial a_0^{m-1} }(a_{N-1}a_N^{d-1}) &
	\frac{\partial^{m-1}}{\partial a_0^{m-1} }(a_N^d)  \\
	
	\frac{\partial^{m-1}}{\partial a_0^{m-2}\partial a_1}(a_0^d)&	\frac{\partial^{m-1}}{\partial a_0^{m-2}\partial a_1 }(a_0^{d-1}a_1) &	\frac{\partial^{m-1}}{\partial a_0^{m-2}\partial a_1 }(a_0^{d-1}a_2) &
	\dots&
	\frac{\partial^{m-1}}{\partial a_0^{m-2}\partial a_1 }(a_{N-1}a_N^{d-1}) &
	\frac{\partial^{m-1}}{\partial a_0^{m-2}\partial a_1 }(a_N^d) \\
	\vdots&\vdots&\vdots& &\vdots&\vdots\\
	\frac{\partial^{m-1}}{\partial a_N^{m-1}}(a_0^d)&	\frac{\partial^{m-1}}{\partial a_N^{m-1} }(a_0^{d-1}a_1) &	\frac{\partial^{m-1}}{\partial a_N^{m-1} }(a_0^{d-1}a_2) &
	\dots&
	\frac{\partial^{m-1}}{\partial a_N^{m-1} }(a_{N-1}a_N^{d-1}) &
	\frac{\partial^{m-1}}{\partial a_N^{m-1}}(a_N^d)
	\end{array}\right].
	$$}
\begin{remark}
    Let $F=F(\a0n)=\det M$. The set $\{F=0\}$ is the locus of such points, for which there exists a non-zero element in the system $\mathfrak{L}(d;mB+Z)$.
\end{remark}
\begin{example}
\label{ex:DogKap} The set $\{F=0\}$ was investigated by \cite{DoKa} in case the points are in general position. In particular, they prove that for any 
$d$, and a set $Z$ of $2d+1$ points in general position, the locus of points $B$, such that the system ${\cal L}(d; (d-1)B+Z)$ is non-empty, is a curve of degree $d(d-1)$ with $(d-1)$-tuple points in $Z$. In Figures \ref{fig:7pts}  and \ref{fig:9pts} we show, for ``general'' points listed below, 
 the locus of the interpolation matrix for $d=3$ and $d=4$. These are exactly the curves described by Dolgachev and Kapranov in \cite{DoKa}. 

The set $Z$ of seven points consists of
    $$Z=\{(0,0,1),(1,0,1),(0,1,1),(2,-2,1),(-1,-3,1),(3,5,1),(4,1,1) \}$$
and to get the set of nine points, we add to $Z$ the following two points $\{(-3,5,1),(-5,2,1) \}$.
    \begin{figure}[ht]
        \centering
        \begin{subfigure}{0.48\textwidth}
            \centering
            \includegraphics[width=0.85\linewidth]{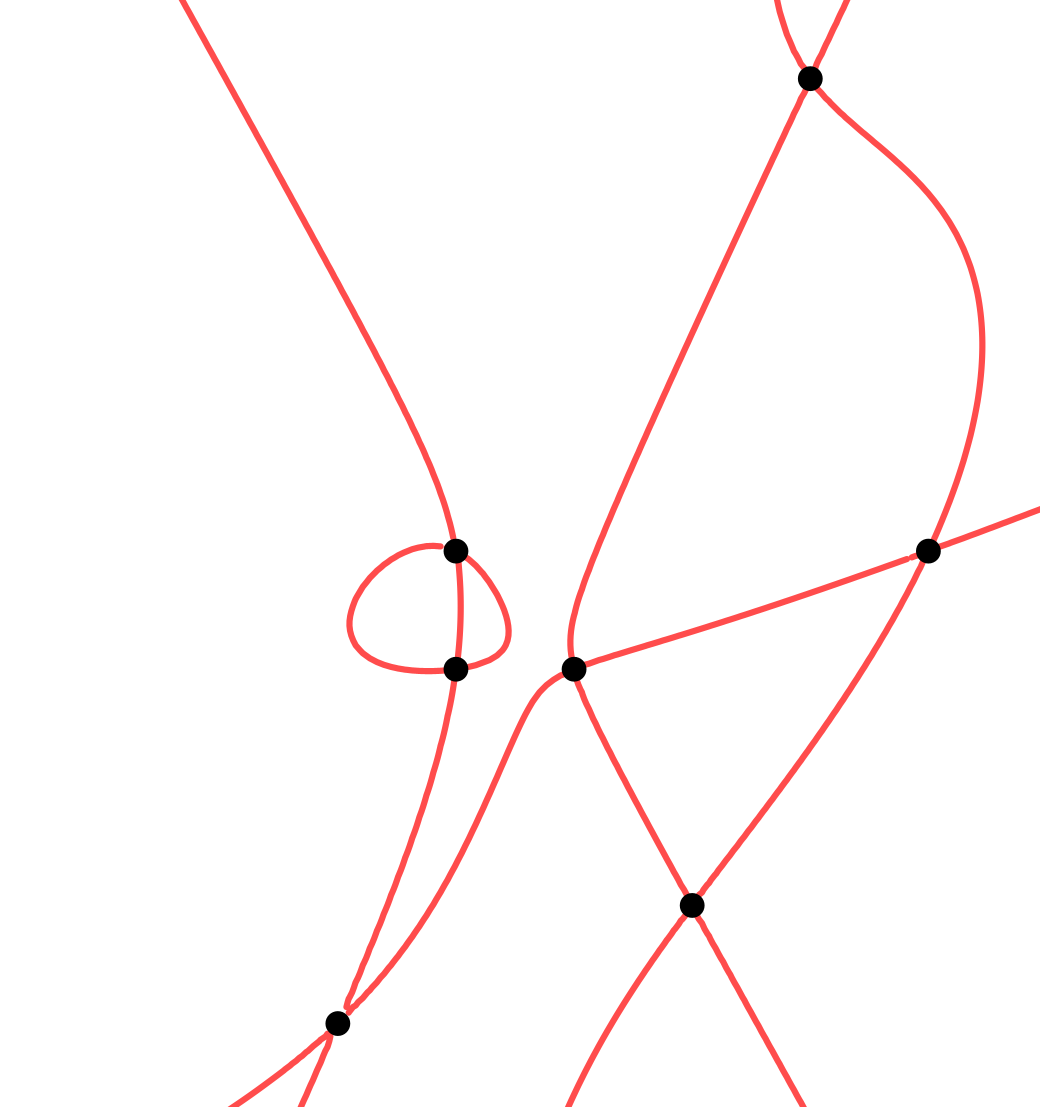}
            \caption{}
            \label{fig:7pts}
        \end{subfigure}
        \hfill 
        \begin{subfigure}{0.48\textwidth}
            \centering
            \includegraphics[width=\linewidth]{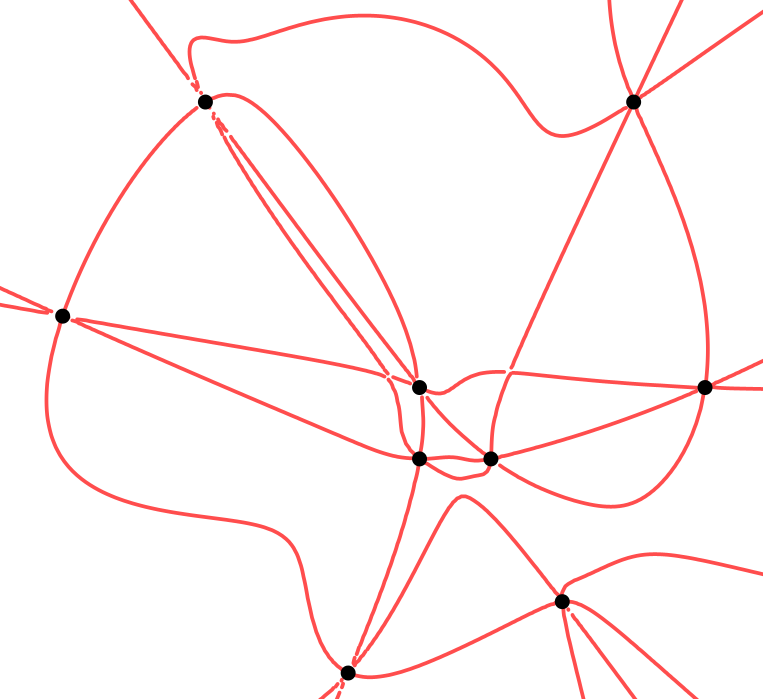}
            \caption{}
            \label{fig:9pts}
        \end{subfigure}
        \caption{The set $\{F=0\}$ for 7 and for 9 general points }
        \label{fig:jumpingLines}
    \end{figure}
\end{example}

Now, we present the essential result of our paper, namely the theorem, which enables us to bound the multiplicity of $B$ in the locus from below.

\begin{theorem}
\label{thm:main}
Let $d$ and $m$ be positive integers with $d\geq m$.
Let $P_1,\dots, P_s$ be pairwise distinct points in $\PP^N$, where $s$ is such that 
$$
\binom{d+N}{N}=\binom{m+N-1}{N}+s.
$$
Let $B$ be a fixed point in $\PP^N$, and $F(\a0n)=\det M$ a polynomial of degree $\binom{m+N-1}{N}(d-m+1)$.
Define
$$h_{j,B}:=
\begin{cases}
    \dim{\cal L}(d;jB+Z)-\left(\binom{d+N}{N}-\binom{j+N-1}{N}-s\right), & \text{ for } j=1,\dots,m \\
    \dim{\cal L}(d;jB+Z), & \text{ for } j=m+1,\dots,d
\end{cases},$$
and $h_B:=\sum_{j=1}^d h_{j,B}.$ Then the multiplicity of $B$ in the set $F=0$ is at least $h_B$. 
\end{theorem}

\begin{notation}
    To simplify the notation, in what follows we will mostly
write $(d; mB+Z)$ instead of ${\cal L}(d; mB+Z)$.
\end{notation}

\begin{proof}
For the proof, similarly to \cite{DFHMST}, we dehomogenize the matrix $M$ from \eqref{eq:defM}, and consequently, the equation of $F$ is also considered in its dehomogenized form. This approach does not change the thesis of the proof. One can use two explanations. The first is to observe that the problem of computing the multiplicity of $F$ at $B$ is local, so the affine version is as good as the homogeneous. The second approach is presented in \cite{DFHMST} and boils down to checking all the details of the dehomogenization.

First, we introduce the necessary notation. Denote by $\widehat{w}(a_1,\ldots,a_N)$ the vector $w$ after dehomogenization with respect to the chosen variable $a_0$, such that the first entry in $P_j$ is scaled and, after possibly linear change of coordinates, equal $1$. We construct an ascending series of affine interpolation matrices $M_j$, for $j=1,\ldots,d$, defined as follows. The first $s$ rows correspond to the values $\widehat{w}(P_1), \ldots, \widehat{w}(P_s)$, followed by $\frac{j\binom{j+N}{N}}{N+1}$ rows, which contain the values of all partial derivatives of $\widehat{w}$ of orders ranging from $0$ to $j-1$. Note that each entry in $M_j$ is an element of $\K[a_1,\ldots,a_n]$ and $M_d$ is the matrix that has $\binom{d+N}{N}$ columns and $\frac{d\binom{d+N}{N}}{N+1}+s$ rows. 

\noindent
\begin{minipage}{0.65\textwidth}
If we denote by $M$ dehomogenized matrix from equation \eqref{eq:defM}, we get the chain of ``inclusions"
$$M_1 \subset M_2 \subset \ldots \subset M_m=M \subset \ldots \subset M_d, $$
which, in fact, means that each matrix $M_j$ consists of the matrix $M_{j-1}$ completed by rows of all partial derivatives of orders $j-1$, and therefore we write $M_{j+1}\setminus M_j$  for the set of derivatives of range $j+1$ in the interpolation matrix $M_d$.
\end{minipage}
\hfill
\begin{minipage}{0.30\textwidth}
\begin{tikzpicture}[line cap=round,line join=round,>=triangle 45,x=1.0cm,y=1.0cm,scale=0.5]
\clip(1.84,-8.5) rectangle (10,4);
\draw [line width=0.5pt] (2.5,4.)-- (2.,4.);
\draw [line width=0.5pt] (2.,4.)-- (2.,2.);
\draw [line width=0.5pt] (2.,2.)-- (2.5,2.);
\draw [line width=0.5pt] (5.5,4.)-- (6.,4.);
\draw [line width=0.5pt] (6.,4.)-- (6.,2.);
\draw [line width=0.5pt] (6.,2.)-- (5.5,2.);
\draw [line width=0.5pt,dash pattern=on 2pt off 2pt] (2.,2.)-- (2.,1.);
\draw [line width=0.5pt,dash pattern=on 2pt off 2pt] (6.,2.)-- (6.,1.);
\draw [line width=0.5pt] (2.,1.)-- (2.,0.);
\draw [line width=0.5pt] (6.,1.)-- (6.,0.);
\draw [line width=0.5pt] (2.,0.)-- (2.5,0.);
\draw [line width=0.5pt] (5.5,0.)-- (6.,0.);
\draw [line width=0.5pt] (2.,-2.)-- (2.,-4.);
\draw [line width=0.5pt] (2.,-6.)-- (2.,-8.);
\draw [line width=0.5pt] (6.,-2.)-- (6.,-4.);
\draw [line width=0.5pt] (6.,-6.)-- (6.,-8.);
\draw [line width=0.5pt] (2.,-8.)-- (2.5,-8.);
\draw [line width=0.5pt] (5.5,-8.)-- (6.,-8.);
\draw (6,2.8) node[anchor=north west] {$M_1$};
\draw (6,0.8) node[anchor=north west] {$M_2$};
\draw (6.2,-1) node[anchor=north west] {$\vdots$};
\draw (6,-3.1) node[anchor=north west] {$M_m=M$};
\draw (6.2,-5) node[anchor=north west] {$\vdots$};
\draw (6,-7.1) node[anchor=north west] {$M_d$};
\draw [line width=0.5pt] (2.,-4.)-- (2.5,-4.);
\draw [line width=0.5pt] (6.,-4.)-- (5.5,-4.);
\draw [line width=0.5pt,dotted] (2.,-4.)-- (2.,-6.);
\draw [line width=0.5pt,dotted] (6.,-4.)-- (6.,-6.);
\draw [line width=0.5pt,dotted] (2.,0.)-- (2.,-2.);
\draw [line width=0.5pt,dotted] (6.,0.)-- (6.,-2.);
\end{tikzpicture}
\end{minipage}

An example of such matrices for $N=2$, $s=11$ and $m=d=3$ is as follows.
\begin{footnotesize}
    \begin{equation*}
        \begin{bmatrix}
            \vdots \\
            \widehat{w}(P_{10}) \\
            \widehat{w}(P_{11}) \\
            \widehat{w} \\
            \partial_{a_1}\widehat{w} \\
            \partial_{a_2}\widehat{w} \\
            \partial_{a_1}\partial_{a_1}\widehat{w} \\
            \partial_{a_1}\partial_{a_2}\widehat{w} \\
            \partial_{a_2}\partial_{a_2}\widehat{w} 
        \end{bmatrix}=
        \begin{bmatrix}
            &&&&&\vdots&&&& \\
            p_{1,10} & p_{2,10} & p_{3,10} & p_{4,10} & p_{5,10} & p_{6,10} & p_{7,10} & p_{8,10} & p_{9,10} & p_{10,10} \\
            p_{1,11} & p_{2,11} & p_{3,11} & p_{4,11} & p_{5,11} & p_{6,11} & p_{7,11} & p_{8,11} & p_{9,11} & p_{10,11} \\
            1 & a_1 & a_2 & a_1^2 & a_1 a_2 & a_2^2 & a_1^3 & a_1^2 a_2 & a_1 a_2^2 & a_2^3 \\
            0 & 1 & 0 & 2a_1 & a_2 & 0 & 3a_1^2 & 2a_1 a_2 & a_2^2 & 0 \\
            0 & 0 & 1 & 0 & a_1 & 2a_2 & 0 & a_1^2 & 2 a_1 a_2 & 3 a_2^2 \\
            0 & 0 & 0 & 2 & 0 & 0 & 6a_1 & 2 a_2 & 0 & 0 \\
            0 & 0 & 0 & 0 & 1 & 0 & 0 & 2a_1 & 2a_2 & 0 \\
            0 & 0 & 0 & 0 & 0 & 2 & 0 & 0 & 2 a_1 & 6a_2
        \end{bmatrix} =
        \left[ \begin{array}{cc}
            \\
            \\
            M_1 \\
           \\ \hline
          \\
            M_2\\ \hline
          \\
            M_3 \\
        \\
        \end{array} \right]
    \end{equation*}
\end{footnotesize}

The main idea of the proof is to show that all derivatives of $F=\det M$ up to order $h_B-1$ vanish at $B$. Since the determinant is a multilinear function, derivative of $F$ can be represented as a sum of determinants of matrices obtained by applying appropriate partial derivatives to the rows of $M$, for example 
\begin{equation}\label{multidet}
\partial \det \left[ \begin{array}{c} \text{row}_1 \\ \text{row}_2 \\ \vdots \\ \text{row}_n \end{array}\right] =
\det \left[ \begin{array}{c} \partial(\text{row}_1) \\ \text{row}_2 \\ \vdots \\ \text{row}_n \end{array}\right] +
\det \left[ \begin{array}{c} \text{row}_1 \\ \partial(\text{row}_2) \\ \vdots \\ \text{row}_n \end{array}\right] + \ldots +
\det \left[ \begin{array}{c} \text{row}_1 \\ \text{row}_2 \\ \vdots \\ \partial(\text{row}_n) \end{array}\right].
\end{equation}
For higher order derivatives we use a similar formula. For example, for a derivative $\partial_{a_1}^3 \partial_{a_2}^2$ the part of this sum may look as follows
$$
\ldots+\det \left[ \begin{array}{c} \partial_{a_1}\partial_{a_1}\partial_{a_2}(\text{row}_1) \\ \partial_{a_2}(\text{row}_2) \\ \vdots \\ \partial_{a_1}(\text{row}_n) \end{array}\right] +
\ldots+
\det \left[ \begin{array}{c} \text{row}_1 \\ \partial_{a_1}\partial_{a_1}\partial_{a_2}(\text{row}_2) \\ \vdots \\ \partial_{a_1}\partial_{a_2}(\text{row}_n) \end{array}\right]+
\ldots \quad.
$$

Consider one term in this sum, denoted by $ \widetilde{M}$. Thus, our aim is to prove that $\det \widetilde{M}(B)=0.$ Observe that $\widetilde{M}$ is obtained by applying partial derivatives to rows of $M$. Since the order of differentiating is not important, we will differentiate from top to bottom, i.e. first we differentiate the rows of $M_1$, then $M_2 \setminus M_1$, and so on.

\vspace{0.2cm}

\noindent
\begin{minipage}{0.65\textwidth}
$\quad$ Stage $1$. Assume that a fixed row of $M_1$ is differentiated $t$ times. Then, after differentiation, ranks of $M_1(B), \ldots,M_t(B)$ may jump by at most one (because only one row is changed). In contrast, ranks of  $M_{t+1}(B), \ldots,M_d(B)$ cannot jump, since $t$-th derivative of this row appears in these matrices. 

$\quad$ We apply all necessary derivatives to rows of $M_1$ obtaining $M_1^{(1)}$ (and $M_j^{(1)}$, since all these matrices contain $M_1^{(1)}$). After applying $t$ derivatives to rows of $M_1$, the sum of ranks may jump by at most $t$:
$$
\sum_{j=1}^d \rank M_j^{(1)}(B) \leq \sum_{j=1}^d \rank M_{j}(B) +t $$
After the described first stage, the resulting matrix is shown in the adjacent figure.
\end{minipage}
\hfill
\begin{minipage}{0.30\textwidth}
\begin{tikzpicture}[line cap=round,line join=round,>=triangle 45,x=1.0cm,y=1.0cm,scale=0.5]
\clip(1.84,-8.5) rectangle (12,4.5);
\draw [line width=0.5pt] (2.5,4.)-- (2.,4.);
\draw [line width=0.5pt] (2.,4.)-- (2.,2.);
\draw [line width=0.5pt] (2.,2.)-- (2.5,2.);
\draw [line width=0.5pt] (5.5,4.)-- (6.,4.);
\draw [line width=0.5pt] (6.,4.)-- (6.,2.);
\draw [line width=0.5pt] (6.,2.)-- (5.5,2.);
\draw [line width=0.5pt,dash pattern=on 2pt off 2pt] (2.,2.)-- (2.,1.);
\draw [line width=0.5pt,dash pattern=on 2pt off 2pt] (6.,2.)-- (6.,1.);
\draw [line width=0.5pt] (2.,1.)-- (2.,0.);
\draw [line width=0.5pt] (6.,1.)-- (6.,0.);
\draw [line width=0.5pt] (2.,0.)-- (2.5,0.);
\draw [line width=0.5pt] (5.5,0.)-- (6.,0.);
\draw [line width=0.5pt] (2.,-2.)-- (2.,-4.);
\draw [line width=0.5pt] (2.,-6.)-- (2.,-8.);
\draw [line width=0.5pt] (6.,-2.)-- (6.,-4.);
\draw [line width=0.5pt] (6.,-6.)-- (6.,-8.);
\draw [line width=0.5pt] (2.,-8.)-- (2.5,-8.);
\draw [line width=0.5pt] (5.5,-8.)-- (6.,-8.);
\draw (6,2.8) node[anchor=north west] {$M_1^{(1)}$};
\draw (7.9,4.5) node[anchor=north west] {Only here};
\draw (8.3,3.6) node[anchor=north west] {changes};
\draw (7.9,2.7) node[anchor=north west] {were made};
\draw (6,0.8) node[anchor=north west] {$M_2^{(1)}$};
\draw (6.2,-1) node[anchor=north west] {$\vdots$};
\draw (6,-3.1) node[anchor=north west] {$M_m^{(1)}$};
\draw (6.2,-5) node[anchor=north west] {$\vdots$};
\draw (6,-7.1) node[anchor=north west] {$M_d^{(1)}$};
\draw [line width=0.5pt] (2.,-4.)-- (2.5,-4.);
\draw [line width=0.5pt] (6.,-4.)-- (5.5,-4.);
\draw [line width=0.5pt,dotted] (2.,-4.)-- (2.,-6.);
\draw [line width=0.5pt,dotted] (6.,-4.)-- (6.,-6.);
\draw [line width=0.5pt,dotted] (2.,0.)-- (2.,-2.);
\draw [line width=0.5pt,dotted] (6.,0.)-- (6.,-2.);
\end{tikzpicture}
\end{minipage}

\vspace{0.2cm}

We will continue in a similar way. Let us describe the general step in our induction.

Stage $k$. After previous stages, we have 
$$
\sum_{j=1}^d \rank M_j^{(k-1)}(B) \leq \sum_{j=1}^d \rank M_{j}(B) + \text{ number of derivatives used so far}. $$
Again, we apply all necessary derivatives to the rows of $ M_k^{(k-1)} \setminus M_{k-1}^{(k-1)}$. If a fixed row is differentiated $t$ times, then the ranks of $M_k^{(k-1)}, \ldots, M_{k+t-1}^{(k-1)}$ may increase by at most one, while the ranks of $M_{k+t}^{(k-1)}, \ldots, M_d^{(k-1)}$ cannot change. We conclude that
\begin{equation}
\label{eq:Mk}
    \sum_{j=1}^d \rank M_j^{(k)}(B) \leq \sum_{j=1}^d \rank M_{j}(B) + \text{ number of derivatives used up to Stage $k$}
\end{equation}

Let $r_j$ denote the minimum of the numbers of rows and columns of $M_j$. Observe that, by definition, for each $j$, the number $h_{j,B}$ represents the deficiency rank of the interpolation matrix $M_j(B)$, so 
\begin{equation}
\label{eq:maxRankEq}
    \sum_{j=1}^d \rank M_j(B) \leq \sum_{j=1}^d r_j -h_B.
\end{equation}

Assume that $\det\widetilde{M}(B)\neq 0$. Thus, all matrices $M_1^{(d)}, \ldots,M_{d}^{(d)}$ must have maximal rank:
\begin{equation}
\label{eq:maxRank}
    \sum_{j=1}^d \rank M_j^{(d)} = \sum_{j=1}^d r_j.
\end{equation}
Combining \eqref{eq:Mk} for Stage $d$, \eqref{eq:maxRankEq} and \eqref{eq:maxRank}, we obtain 
$$\text{ number of derivatives used up to Stage $d$} \geq h_B,$$
which is a contradiction.

\end{proof}
\begin{remark}
    Going back to Example \ref{ex:DogKap} we easily recover the fact that the curve described in \cite{DoKa} has $(d-1)$-tuple points; namely for $B\in Z$ we have $h_{j,B}\geq 1$, for $j=1,\ldots, d-1$, so
    $h_B\geq d-1.$
\end{remark}

\section{Examples with a square interpolation matrix}\label{sec:ex1}
We begin this section by recalling some well-known line configuration on the projective plane, which are simplicial arrangements, i.e., all cells cut out by lines are triangles. 

As the first example, we consider the so-called $A(4k+1,1)$ a family of lines (see e.g. \cite{Gru09}), which are defined over the real projective plane and can be viewed as follows. We consider a regular polygon with $2k\geq 4$ edges, and we take $2k$ lines corresponding to the sides of the $2k$-gone, the lines corresponding to symmetry axes of the $2k$-gone, and the line at infinity. This gives $4k+1$ lines in total in this configuration. In \cite{DMO} the authors prove that the dual configuration of points to $A(4k+1,1)$ gives rise to the unexpected curves of type $(2k,2k-1)$, by using supersolvable property (see \cite[Theorem 3.15]{DMO}). As an example of the application of our main theorem, we present a different proof of this fact.
    
\begin{example}
\label{ex:A4k+1}
    We consider the set $Z\in \PP^2$ of $4k+1$ points, corresponding to the dual of the $A(4k+1,1)$ line arrangement. This set consists of $2k(k-1)$ lines, each passing through exactly $3$ points, $k$ lines intersecting precisely $4$ points, and a single line passing through $2k$ points. We claim that there exists an unexpected curve of type $(d,m)=(2k,2k-1)$. To prove this, it is sufficient to show that this curve exists, and to prove this, we use Theorem \ref{thm:main}. The degree of the polynomial $F$ (appearing in Theorem \ref{thm:main}) is equal to $\binom{2k}{2}\cdot 2=4k^2-2k$, but our aim is to show that $F \equiv 0$.
    
    Let $B=(a_0,a_1,a_2) \in \PP^2$ be in any line $L$, which passes exactly through $3$ points from $Z$.  Then, by Theorem \ref{thm:main}, the multiplicity of $B$ is at least
    \begin{multline*}
        \sum_{j=1}^{2k} h_{j,B} \geq h_{2k-1,B}=\dim(2k;(2k-1)B+Z)-\vdim(2k;(2k-1)B+Z)=\\
        \dim(2k;(2k-1)B+Z).
    \end{multline*}
    By B\'{e}zout's Theorem, each element of the system $\mathfrak{L}(2k;(2k-1)B+Z)$ contains the line $L$. Therefore,
    \begin{equation*}
        \dim(2k;(2k-1)B+Z)=\dim(2k-1;(2k-2)B+Z\setminus L) \geq \vdim(2k-1;(2k-2)B+Z\setminus L)=1,
    \end{equation*}
    which means that every line of this kind has to appear in the equation of $F$ at least once. 
    Let the point $B$ lie on a line $L$ that passes through exactly $4$ points from $Z$. We have
    \begin{equation*}
        \sum_{j=1}^{2k} h_{j,B}\geq h_{2k-2,B}+h_{2k-1,B}.
    \end{equation*}
    Observe that again from B\'{e}zout's Theorem, each element of the system $\mathfrak{L}(2k;(2k-2)B+Z)$ (resp. $\mathfrak{L}(2k;(2k-1)B+Z)$) contains the line $L$, thus
    \begin{multline*}
        h_{2k-2,B}=\dim(2k;(2k-2)B+Z)-\vdim(2k;(2k-2)B+Z) \geq \\
        \vdim(2k-1;(2k-3)B+Z\setminus L)-\vdim(2k;(2k-2)B+Z)=1.
    \end{multline*}
    and
    \begin{multline*}
        h_{2k-1,B}=\dim(2k;(2k-1)B+Z)-\vdim(2k;(2k-1)B+Z) =\dim(2k;(2k-1)B+Z) \geq \\
        \vdim(2k-1;(2k-2)B+Z\setminus L)=2.
    \end{multline*}
    So, any line of this type occurs in $F$ with multiplicity at least $3$.
    Finally, for a point $B$ on a line $L$, which contains $2k$ points, we have
    \begin{equation*}
        \sum_{j=1}^{2k} h_{j,B} \geq h_{2,B}+\ldots+h_{2k-1,B} \geq 1+2+\ldots+2k-2=\binom{2k-1}{2}.
    \end{equation*}
    where each of the number $h_{j,B}$, with $j=2,\ldots,2k-1$, can be obtained by proceeding in the same way, as we showed in the previous cases, i.e.
    \begin{multline*}
        h_{j,B}=\dim(2k;jB+Z)-\vdim(2k;jB+Z) \geq \vdim(2k-1;(j-1)B+Z\setminus L)-\vdim(2k;jB+Z) \\
        =j-1.
    \end{multline*}
    
If we sum all multiplicities of lines obtained in all cases, we can calculate the minimal degree of $F$ (to contain all the lines), which is
$$2k(k-1)\cdot 1+k \cdot 3+ 1 \cdot \binom{2k-1}{2}=4k^2-2k+1,$$
more than its actual degree. Thus $F \equiv 0$, as desired.
\end{example}
Example \ref{ex:A4k+1} shows an alternative way to prove what was known. For $k=1$, Example \ref{ex:A4k+1}  gives a purely combinatorial method of showing the existence of the famous quartic from \cite{CHMN16}, already mentioned in the Introduction. Such a quartic exists for 9 points in such a configuration that there are 3 lines with 4 points and 4 lines with 3 points. Then, the proof (with the use of Theorem \ref{thm:main}) boils down to checking that $3\cdot 3 +4 > 12$.
\begin{remark}
\label{rmk:hjBLine}
    There is an alternative way to obtain a lower bound for the values of the numbers $h_{j,B}$ , where $B$ is a point lying on a line $L$. This approach arises from the application of Bézout's Theorem. We can prove that $h_{j,B}$ is greater or equal to the number of points in excess required for $L$ to be an element of the system $\mathfrak{L}(d;jB+Z)$. If we denote by $|L|$ the number of points in the set $Z \cap L$, then
    \begin{multline*}
    \label{eq:hjB}
        h_{j,B}=\dim(d;jB+Z)-\vdim(d;jB+Z) \geq \\
        \vdim\left(d-1;(j-1)B+Z\setminus L\right)-\vdim(d;jB+Z) =
        j+ |L| -d -1,
    \end{multline*}
    for all $j$ which fulfil inequalities $m \geq j \geq d+2 - |L|$. Therefore, we have
    \begin{equation*}
    \label{eq:hBForL}
        \sum_{j=1}^m h_{j,B} \geq 1+2+\ldots+(m-d-2+|L|+1)=\binom{|L|+m-d}{2}.
    \end{equation*}
\end{remark}
\begin{example}
    \label{eq:FastExample}
    Consider another family of lines from Grünbaum's list \cite{Gru09}, denoted there as $A(13,3)$. As the set $Z$, we take the dual set of points corresponding to this family of lines. In \cite{DMO}, the authors claimed that the set $Z$ admits an unexpected curve of degree $d=6$ with a general point of multiplicity $m=5$. Their argument is based on computational calculations performed using the {\tt Singular} program. We apply Remark \ref{rmk:hjBLine} to provide a non-computer-based proof of the existence of this curve.

    Using Theorem \ref{thm:main}, we claim that $\deg F=\binom{6}{2}\cdot 2=30$. There are $10$, $3$, and $2$ lines passing through $3$, $4$, and $5$ points in $Z$, respectively. Thus, by Remark \ref{rmk:hjBLine}, the equation $F=0$ contains a set of lines whose total sum of multiplicities is  
    $$
    10\cdot \binom{2}{2}+3\cdot \binom{3}{2}+2\cdot \binom{4}{2}=31 > \deg F,
    $$
    which implies that $F \equiv 0$.
\end{example}
\begin{remark}
    We have prepared an external file that contains a constantly updated list of known curves and unexpected surfaces \cite{github}. In this file, one can find arrangements of lines, along with information about their weak combinatorics and whether our methods apply to the proof of a given case. All data about weak combinatorics or exponents of the arrangements listed in the file comes from \cite{Cuntz, Gru09, Janasz}. At the time of this paper's publication, $29$ examples and two infinite families of line arrangemet on the list had been proved using the method indicated in Example \ref{eq:FastExample}.     
\end{remark}
\begin{remark}
\label{rmk:hjBFixedPoint}
    We can provide a similar characterization of the lower bound for the values of the numbers $h_{j,B}$ when the point $B$ is fixed. To achieve this, let $|L|$ denote the number of points lying on a line $L$ that intersects the set $Z$.

    Next, suppose $B \not\in Z$ and that there exist $k \geq 1$ lines $L_i $ such that $B \in L_i$. Then, for any $j$, such that $m \geq j \geq d - |L_i| + 2 $, the following inequality holds
    \begin{footnotesize}
        \begin{equation*}
        h_{j,B}=\dim(d;jB+Z)-\vdim(d;jB+Z) \geq \vdim\left(d-k; (j-k)B+Z \setminus \bigcup_{i=1}^k (L_i \cap Z)\right)-\vdim(d;jB+Z).
        \end{equation*}
    \end{footnotesize}
    After performing basic computations, this simplifies to
    \begin{footnotesize}
        \begin{multline}
        \label{eq:sumFor1B}
        \frac{(d+2)(d-1)-k(2d+3)+k^2}{2}-\frac{(j+1)j-k(2j+1)+k^2}{2}-|Z|+\sum_{i=1}^k |L_i| \\
        - \frac{(d+2)(d-1)}{2}+\frac{(j+1)j}{2}+|Z|=k(j-d-1)+\sum_{i=1}^k |L_i|.
        \end{multline}
    \end{footnotesize}
    
    If $k=1$, then by formula \eqref{eq:sumFor1B}, we have
    $$
    \sum_{j=1}^m h_{j,B} = \sum_{j=d-|L|+2}^m h_{j,B} \geq 1+2+ \ldots +(m-d-1+|L|) = \binom{|L|+m-d}{2},
    $$
    which is exactly the same number as we obtain in Remark \ref{rmk:hjBLine} for points $B$ lying on the line $L$.
    
    For the case where $B\in Z \cap L_i$ and $j$ is within the range $m \geq j \geq d-|L_i|+3$, we carry out analogous computations. Specifically
    \begin{footnotesize}
        \begin{multline}
        \label{eq:sumFor1BWithB}
        h_{j,B} = \dim(d;jB+Z) - \vdim(d; jB + Z) \geq \\
        \vdim\left(d-k; (j-k)B+Z \setminus \bigcup_{i=1}^k (L_i \cap Z)\right)-\vdim(d;jB+Z)=k(j-d-2) + \sum_{i=1}^k |L_i| + 1.
        \end{multline}
    \end{footnotesize}
    Observe that for $j \in \{1,2,\ldots,d-|L_i|\}$, we always have a lower bound $h_{j,B} \geq 1$, as point $B$ is one of the points from $Z$, and therefore the matrix $M(B)$ contains two linearly dependent rows.

\end{remark}

        \section{Examples with a rectangular interpolation matrix}\label{sec:ex2}

In the next example we show a broader application of Theorem \ref{thm:main}. In this example, in order to show the existence of an unexpected curve, we need to compute the value $h_{j,B},$ for $j=m+1=d$. 
\begin{example}
\label{ex:sporadic}
Consider the dual set of points $Z$ for the so-called sporadic simplicial arrangements $A(30,3)$ (see e.g. \cite{Gru09}), which is shown in the adjacent figure. This set consists of $30$ points, $29$ black, and one highlighted red open circle. In order not to make the image unreadable, only some collinearities have been presented as dotted lines. Note that the circle represents the line at infinity.

\noindent
\begin{minipage}{0.50\textwidth}
 In \cite{DMO} authors claimed that set $Z$ admits unexpected curve in degree $d=14$ with a general point of multiplicity $m=13$. Their argument is based on computational calculations made in {\tt Singular} program. Theorem \ref{thm:main} provide tools to show existence of unexpected curve in this case, without the aid of a computer. As before, the proof is based only on the week combinatoric of $Z$.

 At the beginning we observe that the required number of points, for the matrix $M$ to be a square matrix, is
 $$
s=\binom{16}{2}-\binom{14}{2}=29,
$$
so we need to consider every subset $Z'$ of $Z$, which contains $29$ points, and show that for every such chose $\det M =F \equiv 0$.
\end{minipage}
\hfill
\begin{minipage}{0.45\textwidth}
   \begin{tikzpicture}[line cap=round,line join=round,>=triangle 45,x=1.0cm,y=1.0cm, scale=4]
\clip(-0.9,-0.8239966106878052) rectangle (0.8157985585252129,0.8892647506238589);
\draw [line width=0.4pt,dotted] (0.,0.) circle (0.8cm);
\draw [line width=0.4pt,dotted,domain=-0.7987933003025381:0.8157985585252129] plot(\x,{(-0.--0.3849001794597505*\x)/0.6666666666666666});
\draw [line width=0.4pt,dotted,domain=-0.7987933003025381:0.8157985585252129] plot(\x,{(-0.-3.*\x)/-1.7320508075688772});
\draw [line width=0.4pt,dotted,domain=-0.7987933003025381:0.8157985585252129] plot(\x,{(-0.-0.3849001794597505*\x)/0.6666666666666666});
\draw [line width=0.4pt,dotted,domain=-0.7987933003025381:0.8157985585252129] plot(\x,{(-0.-0.5773502691896258*\x)/0.3333333333333333});
\draw [line width=0.4pt,dotted] (0.,-0.8239966106878052) -- (0.,0.8892647506238589);
\draw [line width=0.4pt,dotted,domain=-0.7987933003025381:0.8157985585252129] plot(\x,{(-0.-0.*\x)/0.8});
\begin{scriptsize}
\draw [fill=black] (0.,0.) circle (0.5pt);
\draw [fill=black] (-0.16666666666666666,0.2886751345948129) circle (0.5pt);
\draw [fill=black] (0.16666666666666666,0.2886751345948129) circle (0.5pt);
\draw [fill=black] (0.16666666666666666,-0.2886751345948129) circle (.5pt);
\draw [fill=black] (-0.16666666666666666,-0.2886751345948129) circle (0.5pt);
\draw [fill=black] (0.16666666666666666,-0.09622504486493763) circle (0.5pt);
\draw [fill=black] (-0.16666666666666666,-0.09622504486493763) circle (0.5pt);
\draw [fill=black] (0.16666666666666666,0.09622504486493763) circle (.5pt);
\draw [fill=black] (-0.16666666666666666,0.09622504486493763) circle (0.5pt);
\draw [fill=black] (-0.3333333333333333,0.) circle (0.5pt);
\draw [fill=black] (0.3333333333333333,0.) circle (0.5pt);
\draw [fill=black] (0.,0.19245008972987526) circle (0.5pt);
\draw [fill=black] (0.,-0.19245008972987526) circle (0.5pt);
\draw [fill=black] (-0.3333333333333333,0.19245008972987526) circle (0.5pt);
\draw [fill=black] (-0.3333333333333333,-0.19245008972987526) circle (0.5pt);
\draw [fill=black] (0.3333333333333333,-0.19245008972987526) circle (0.5pt);
\draw [fill=black] (0.3333333333333333,0.19245008972987526) circle (0.5pt);
\draw [fill=black] (0.,0.09622504486493763) circle (0.5pt);
\draw [fill=black] (0.,-0.09622504486493763) circle (0.5pt);
\draw [fill=black] (0.08333333333333333,-0.048112522432468816) circle (0.5pt);
\draw [fill=black] (0.08333333333333333,0.048112522432468816) circle (0.5pt);
\draw [fill=black] (0.08333333333333333,0.048112522432468816) circle (0.5pt);
\draw [fill=black] (-0.08333333333333333,-0.048112522432468816) circle (0.5pt);
\draw [fill=black] (-0.08333333333333333,0.048112522432468816) circle (0.5pt);
\draw [fill=black] (0.,-0.3849001794597505) circle (0.5pt);
\draw [fill=black] (0.,0.8) circle (0.5pt);
\draw [fill=black] (0.8,0.) circle (0.5pt);
\draw [fill=black] (0.4,0.692820323027551) circle (0.5pt);
\draw [color=red] (0.6928203230275511,0.4) circle (0.7pt);
\draw [fill=black] (0.4,-0.692820323027551) circle (0.5pt);
\draw [fill=black] (0.6928203230275511,-0.4) circle (0.5pt);
\end{scriptsize}
\end{tikzpicture}
\end{minipage}
\vspace{0.1cm}

Consider the subset $Z'$ consisting of $29$ points, represented as black filled dots in the adjacent figure. We use Remark \ref{rmk:hjBLine} to obtain a lower bound for the values of $h_{j,B}$, where $B$ is a point lying on a line that contains at least $3$ points from $Z'$. For this purpose, we present a table that shows the number of lines $L_i$ passing through exactly $i$ points from the subset $Z'$, along with a lower bound for the sum of values of $h_{j,B}$ ($j=1,2,\ldots,13$), for each $i$-point line $L_i$.
\begin{table}[H]
\centering
    \begin{tabular}{|c|r|r|r|r|r|r|}
        \hline
         $i$ & $3$  & $4$  & $5$ & $6$ & $7$ & $8$ \\ \hline
         $\# L_i$ & $42$ & $18$ & $5$ & $0$ & $2$ & $1$ \\ \hline
         $\displaystyle \sum_{j=1}^{13} h_{j,B} \geq$ & $1$ & $3$ & $6$ & $10$ & $15$ & $21$ \\ \hline
    \end{tabular}
\end{table}
\noindent
The degree of $F$, as derived from Theorem \ref{thm:main}, is $\binom{14}{2}\cdot 2=182$. However, from the same theorem and the presented table, we also observe that $F$ consists of lines $L_i$, and the total sum of their multiplicities is given by
$$ 42 \cdot 1 + 18 \cdot 3 + 5 \cdot 6 + 2 \cdot 15 + 1 \cdot 21 = 177. $$

Now, let $B$ be in the position of the point removed from the set $Z$, which is marked as a red open circle in the figure. What we will show is that at this point, the order of vanishing is even greater than what we had previously determined. To see this, first observe that the $8$ lines passing through this point contain all the points from the set $Z'$. Since $B \not \in Z$, by Remark \ref{rmk:hjBFixedPoint}, we can use the same lower bound for the sum $\sum_{j=1}^{13} h_{j,B}$, as we obtained from the lines $L_i$ passing through $B$, such that $|L_i\cap Z'| \geq 3$. Therefore, by Remark \ref{rmk:hjBLine}
$$\sum_{j=1}^{13} h_{j,B} \geq \sum_{i=1}^{5} \binom{|L_i|-1}{2}=1+3+3+6+15=28.$$
Finally, we calculate 
$$h_{14,B}=\dim(14;14B+Z') \geq \dim(6; 6B)=7,$$
but since we cannot realize an additional multiplicity of $7$ at $B$ with a curve of degree $182-177=5$, we conclude that $F \equiv 0$, as desired.

To complete the proof of existence, we need to perform similar calculations for the remaining $29$ possible choices of the subset $Z'$. To avoid these tedious calculations, it is more convenient to use the following theorem.
\end{example}

\begin{theorem}
\label{thm:Zminus1pt}
    Let $d$ and $m=d-1$ be positive integers and let $P_1,\dots,P_s$ be pairwise distinct points in $\PP^2$, where $s=\binom{d+2}{2}-\binom{d}{2}+1$. Let $Z=\{P_1,\dots,P_s\}$ and let $|L|$ denote the number of points lying on a line intersecting the set $Z$. Then $Z$ admits an unexpected curve of type $(d,d-1)$ if
    $$
    \sum_{L_i} \binom{|L_i|-1}{2}+d-s+2 > 2\binom{d}{2},
    $$
    where the sum is over all lines $L_i$ such that $Z \cap L_i \neq \emptyset$ and $|L_i| \geq 3$.
\end{theorem}
\begin{proof}
The number of points in $Z$ is exactly one more than required. Let $P\in Z$ and consider the set $Z'=Z\setminus P$. The aim of the proof is to show that the total multiplicity of lines that $F=0$ must contain, combined with the additional multiplicity at $P$ which has to be realized by $F=0$, exceeds the expected degree of $F$.

 By Theorem \ref{thm:main} and Remark \ref{rmk:hjBLine}, for each line $L_i$, we have that the total multiplicity of lines contained in $F=0$ is greater or equal to
\begin{equation}
\label{eq:hjBLi}
    \sum_{L_i\not\ni P}\binom{|L_i| - 1}{2} + \sum_{L_i \ni P}\binom{|L_i|-2}{2} 
    = \sum_{L_i} \binom{|L_i|-1}{2} - \sum_{L_i \ni P} (|L_i| - 2).
\end{equation}

Now, let $B=P$. Since $B\not\in Z$, Remark \ref{rmk:hjBFixedPoint} allows us to apply the same lower bound to the sum $\sum_{j=1}^{m} h_{j,B}$ as the one derived from the lines $L_i$ passing through $B$. Consequently, we estimate $h_{d,B}$ where $F$ may have additional multiplicity, not covered by multiplicities of lines. By B\'{e}zout's Theorem, each element of the system $\mathfrak{L}(d;dB+Z \setminus B)$ contains lines $L_i$ such that $B \in L_i$. As a result,
\begin{footnotesize}
    \begin{align}
    \label{eq:hdBThm1pt}
        h_{d,B} &= \dim(d;dB+Z\setminus B) \geq \dim \left(d-\sum_{L_i \ni B} 1; \; dB+(Z \setminus \bigcup_{i=1}^k (L_i \cap Z))\cup B \right) \geq \\
        &\vdim \left(d-\sum_{L_i \ni B} 1; \; dB+(Z \setminus \bigcup_{i=1}^k (L_i \cap Z))\cup B \right)= d-\sum_{L_i \ni B} 1 - \left(s - \sum_{L_i \ni B} (|L_i|-1) - 1\right) + 1. \nonumber
    \end{align}
\end{footnotesize}
We see that if the sum \eqref{eq:hjBLi} and \eqref{eq:hdBThm1pt}, which is
\begin{multline*}
    \sum_{L_i} \binom{|L_i|-1}{2} - \sum_{L_i \ni B} (|L_i|-2) 
    + d - \sum_{L_i \ni B} 1 - \left(s - \sum_{L_i \ni B} (|L_i|-1)-1\right) +1 = \\
    \sum_{L_i} \binom{|L_i|-1}{2}+d-s+2,
\end{multline*}
is greater than $\deg F=2 \binom{d}{2}$, then $F\equiv 0$, as desired.
\end{proof}
\begin{remark}
    We apply Theorem \ref{thm:Zminus1pt} to conclude the argument in Example \ref{ex:sporadic}. The necessary combinatorial data are summarized in the table below, which shows the number of lines $L_i$ passing through exactly $i$ points
    \begin{table}[H]
    \centering
        \begin{tabular}{|c|r|r|r|r|r|r|}
            \hline
             $i$ & $3$  & $4$  & $5$ & $6$ & $7$ & $8$ \\ \hline
             $\# L_i$ & $44$ & $17$ & $6$ & $1$ & $1$ & $2$ \\ \hline
        \end{tabular}
    \end{table}
    \noindent
    Therefore, instead of examining all $29$ remaining possibilities, it suffices to verify whether
    \begin{equation*}
        \sum_{L_j} \binom{|L_j|-1}{2}+d-s+2 = 198 +14 -29 +2=185 > \deg F = 182.
    \end{equation*}
    This confirms that the inequality holds, completing the proof.
\end{remark}
\begin{example}
    Consider another line arrangement from the Gr\"unbaum list of simplicial arrangements \cite{Gru09}, namely $A(15,1)$. Denote by $Z$ the set of points dual to these lines. What we can prove is that $Z$ admits three different types of unexpected curves, of types $(6,5)$, $(7,6)$ and $(8,7)$.
    \vspace{0.2cm}

    \noindent
    \begin{minipage}{0.75\textwidth}
        \quad The following table shows the number of lines $L_i$ passing through exactly $i$ points ($i \geq 3$) from the set $Z$. In addition, every point from $Z$ lies on $2$ lines passing through $5$ points and $2$ lines passing through $3$ points.
    \end{minipage}
    \hfill
    \begin{minipage}{0.20\textwidth}
        \begin{center}
            \begin{tabular}{|c|r|r|}
                \hline
                 $i$ & $3$  &  $5$ \\ \hline
                 $\# L_i$ & $10$ & $6$ \\ \hline
            \end{tabular}
        \end{center}
    \end{minipage}

    \begin{description}
        \item [Case: $(d,m)=(6,5)$.] After removing one point from $Z$, we have the following weak combinatorics, shown in the adjacent table.

        \noindent
        \begin{minipage}{0.75\textwidth}
            We check that 
            $$\sum_{i=3}^5 \# L_i \binom{i-1}{2} +d-|Z|+2 =8+2\cdot 3+4 \cdot 6+6-14+2 = 32 $$ is greater than $2\cdot \binom{d}{2}=30$. By Theorem \ref{thm:Zminus1pt}, we obtain the assertion.
        \end{minipage}
        \hfill
        \begin{minipage}{0.20\textwidth}
            \begin{center}
                \begin{tabular}{|c|r|r|r|}
                    \hline
                     $i$ & $3$  &  $4$ & $5$ \\ \hline
                     $\# L_i$ & $8$ & $2$ & $4$\\ \hline
                \end{tabular}
            \end{center}
        \end{minipage}
        \item [Case: $(d,m)=(7,6)$.] In this case, it is sufficient to check, using Remark \ref{rmk:hjBLine}, that the total sum of the multiplicities of the lines contained in $F=0$ is
        $$\sum_{i=3}^5 \# L_i \binom{i-1}{2}= 10+6\cdot 6=46,$$
        which is bigger than $\deg F= 2\cdot \binom{d}{2}=42.$
        \item [Case: $(d,m)=(8,7)$.] In this case, $\vdim(8;7B+Z) = 2$, so it is sufficient to show that a curve exists when two additional points, $P_1$ and $P_2$, are added to the set $Z$. Such a curve does exist and consists of the unexpected curve of type $(6,5)$, along with two additional lines passing through $B$ and one of the points $P_i$.
    \end{description}
\end{example}
\begin{remark}
    Additional $5$ examples of using Theorem \ref{thm:Zminus1pt} can be found in our external file \cite{github}. 
\end{remark}

\section{Example of unexpected curves of type $(d+k, d)$, $k>1$}\label{sec:ex3}

 The next example shows the existence of two unexpected curves of type $(d+k, d)$, where $k>1$, for a kind of Fermat configuration. The curves were found experimentally and mentioned in \cite{KS,MT-G}.
An interesting fact here is that apart from the weak combinatorics of this arrangement (i.e., the points-lines connections) we use more, see 
a) and b) below.

\begin{example}\label{ex:F6-F3}
Let $F_m$ denote the Fermat configuration of points, which consists of the intersection points of the lines defined by the equation 
$$(x^m - y^m)(y^m - z^m)(z^m - x^m) = 0,$$
together with the set $T$ of three fundamental points. Our main focus here is on the configurations $F_3$ and $F_6$, where $F_3\subsetneq F_6$.

Consider a point $P \in F_3\setminus T$, and let $l_1,l_2,l_3$ be  three lines from the equation 
$$(x^3 - y^3)(y^3 - z^3)(z^3 - x^3) = 0,$$
that pass through $P$. Define $Z_1 \subset F_6$ as a configuration consisting of $20$ points obtained by removing from $F_6$ all points lying on these three lines. Furthermore, let $Z$ be a configuration of $30$ points taken from $(F_6\setminus F_3)\cup T$.

In the subsequent discussion, we use the following fact:
\begin{itemize}
    \item[a)] There exists a quartic $K$ passing through $Z_1$.
    \item[b)] Additionally, there is a $4$-dimensional family of quintics passing through $Z_1$.
\end{itemize}
The quartic in $(a)$ was found numerically. The family in $(b)$ can be constructed by taking $K$ and a line (this would give a 3-dimensional family), and by taking also 5 lines containing $Z_1$, see Figure \ref{fig:F6F3}.

\noindent
\begin{figure}
  \begin{minipage}[c]{0.50\textwidth}
    \definecolor{qqzzff}{rgb}{0.,0.6,1.}
    \definecolor{ffqqqq}{rgb}{1.,0.,0.}
    \begin{tikzpicture}[line cap=round,line join=round,>=triangle 45,x=1.0cm,y=1.0cm,scale=0.8]
    \clip(-3.49,-3.31) rectangle (7.2,7.2);
    \draw [line width=0.4pt] (1.,5.5)-- (1.,-0.5);
    \draw [line width=0.4pt] (-0.5,1.)-- (5.5,1.);
    \draw [line width=0.4pt] (-0.25,-0.25)-- (5.5,5.5);
    \draw [line width=0.4pt, dotted] (5.5,5.5)-- (6.75,6.75);
    \draw [line width=0.4pt] (7.25,7.25)-- (6.75,6.75);
    \draw [line width=0.4pt] (-3.25,1)-- (-2.75,1);
    \draw [line width=0.4pt, dotted] (-2.75,1)-- (0,1);
    \draw [line width=0.4pt] (1,-2.75)-- (1,-3.25);
    \draw [line width=0.4pt, dotted] (1,-2.75)-- (1,-0.5);
    \begin{scriptsize}
    \draw (0.95,1.05) node[anchor=north west] {$P$};
    \draw [fill=black] (0.,0.) circle (1.5pt);
    \draw [fill=black] (0.,1.) circle (1.5pt);
    \draw [fill=black] (0.,2.) circle (1.5pt);
    \draw [fill=black] (0.,3.) circle (1.5pt);
    \draw [fill=black] (0.,4.) circle (1.5pt);
    \draw [fill=black] (0.,5.) circle (1.5pt);
    \draw [color=ffqqqq] (1.,5.) circle (2.0pt);
    \draw [fill=black] (1.,4.) circle (1.5pt);
    \draw [color=ffqqqq] (1.,3.) circle (2.0pt);
    \draw [fill=black] (1.,2.) circle (1.5pt);
    \draw [color=ffqqqq] (1.,1.) circle (2.0pt);
    \draw [fill=black] (1.,0.) circle (1.5pt);
    \draw [fill=black] (2.,0.) circle (1.5pt);
    \draw [fill=black] (2.,1.) circle (1.5pt);
    \draw [fill=black] (2.,2.) circle (1.5pt);
    \draw [fill=black] (2.,3.) circle (1.5pt);
    \draw [fill=black] (2.,4.) circle (1.5pt);
    \draw [fill=black] (2.,5.) circle (1.5pt);
    \draw [color=ffqqqq] (3.,5.) circle (2.0pt);
    \draw [fill=black] (3.,4.) circle (1.5pt);
    \draw [color=ffqqqq] (3.,3.) circle (2.0pt);
    \draw [fill=black] (3.,2.) circle (1.5pt);
    \draw [color=ffqqqq] (3.,1.) circle (2.0pt);
    \draw [fill=black] (3.,0.) circle (1.5pt);
    \draw [fill=black] (4.,0.) circle (1.5pt);
    \draw [fill=black] (5.,0.) circle (1.5pt);
    \draw [fill=black] (4.,1.) circle (1.5pt);
    \draw [color=ffqqqq] (5.,1.) circle (2.0pt);
    \draw [fill=black] (5.,2.) circle (1.5pt);
    \draw [fill=black] (4.,2.) circle (1.5pt);
    \draw [fill=black] (4.,3.) circle (1.5pt);
    \draw [color=ffqqqq] (5.,3.) circle (2.0pt);
    \draw [fill=black] (4.,4.) circle (1.5pt);
    \draw [fill=black] (5.,4.) circle (1.5pt);
    \draw [fill=black] (4.,5.) circle (1.5pt);
    \draw [color=ffqqqq] (5.,5.) circle (2.0pt);
    \draw [fill=qqzzff] (1,-3.) ++(-2.5pt,0 pt) -- ++(2.5pt,2.5pt)--++(2.5pt,-2.5pt)--++(-2.5pt,-2.5pt)--++(-2.5pt,2.5pt);
    \draw [fill=qqzzff] (-3.,1) ++(-2.5pt,0 pt) -- ++(2.5pt,2.5pt)--++(2.5pt,-2.5pt)--++(-2.5pt,-2.5pt)--++(-2.5pt,2.5pt);
    \draw [fill=qqzzff] (7.,7.) ++(-2.5pt,0 pt) -- ++(2.5pt,2.5pt)--++(2.5pt,-2.5pt)--++(-2.5pt,-2.5pt)--++(-2.5pt,2.5pt);
    \end{scriptsize}
    \end{tikzpicture}
  \end{minipage}\hfill
  \begin{minipage}[c]{0.40\textwidth}
    \caption{
       In the adjacent figure, the points from the set $F_6$ are represented as black filled dots, while the points from $F_3$ are shown as red open circles. The points belonging to the set $T$ are depicted as blue squares. Additionally, for a representative point $P$, the lines $l_1, l_2, l_3$ passing through $P$ are also illustrated in the figure.
    } \label{fig:F6F3}
  \end{minipage}
\end{figure}
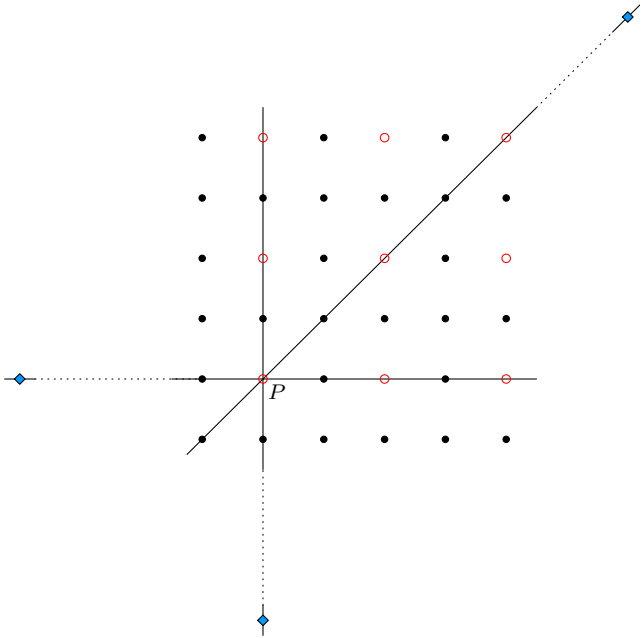
The points in $Z$ are classified into three types: points of type $T$, which belong to the set $T$; points of type $S$, which lie on the lines defining $F_6$ but not $F_3$; and points of type $W$, for which at least one of the defining lines passing through them is associated with $F_3$. It is also worth noting that the points in $Z$ lie on two types of lines: lines of type $L_7$, which contain $7$ points from $Z$, and lines of type $L_4$, which contain $4$ points from $Z$. 

Our goal is to establish the following claim:  
\begin{claim}  
    There exists an unexpected curve of type $(7,3)$ passing through $Z$.  
\end{claim}  
\begin{proof}
    Let $B$ be a point on the line of type $L_7$. Then, by Remark \ref{rmk:hjBLine}, we have 
    $$\sum_{j=1}^3 h_{j,B} \geq \binom{7+3-7}{2}=3.$$ 
    From Theorem \ref{thm:main}, we have that $\deg F = 30$. Thus, $\{F=0\}$ consists of $9$ lines, with each line of type $L_7$ counted three times, and a cubic $C$, or $F \equiv 0$.

    Now, place the point $B$ at any point of type $T$. By Remark \ref{rmk:hjBFixedPoint}, we have that $h_{1,B}, h_{2,B}$ are both greater or equal $1$, and $h_{3,B}\geq 3 \cdot (3-7-2) + 21+1 = 4$. By B\'{e}zout's Theorem, all three lines $L_7$ passing through $B$ are elements of the system $\mathfrak{L}(7; 4B+Z)$. Let $Z'$ denote the set of points from $Z$ after removing these three lines. It follows that
    $$h_{4,B}=\dim(7; 4B+Z) \geq \dim(4; 1B+Z') \geq \vdim(4; 1B+Z')=3.$$
    Since three lines of type $L_7$ and three of type $L_4$ passing through $B$ are elements of $\mathfrak{L}(7; 5B+Z)$, we can use a similar argument to deduce
    $$h_{5,B}=\dim(7; 5B+Z) \geq \dim(1; 0B+Z'') \geq \vdim(1; 0B+Z'')=1,$$
    and analogously $h_{6,B} \geq 1,$
    where $Z''=T\setminus \{B\}.$ Therefore, $h_B \geq 11$, and by Theorem \ref{thm:main}, each point of type $T$ is a double point on the cubic $C$ and this cubic equals to three lines through the points of $T$. So far, if $F\not\equiv 0$,  we know  $F$.

    Take the point $P$ from $F_3$ and a curve of degree seven, consisting of lines $l_1,l_2,l_3$ and the quartic $K.$ This septic has a triple point in $P$, so $P$ is an element of the locus, a contradiction, as $P$ lies neither on lines $L_7$ nor on $C$. Thus, $F\equiv 0$.
\end{proof}

\begin{question}
    Is it possible to find another proof of the Theorem, combinatorial and not using the existence of $K$.
\end{question}

\begin{claim}  
    There exists an unexpected curve of type $(8,5)$ passing through $Z$.  
\end{claim}
\begin{proof}
    Let $B$ be a point on a line of type $L_7$. According to Remark \ref{rmk:hjBLine}, we have:  
    $$\sum_{j=1}^5 h_{j,B} \geq \binom{7+5-8}{2} = 6.$$ 
    It follows that $\{F = 0\}$ consists of $9$ lines, where each line of type $L_7$ is counted six times, along with a sextic curve $S$, or $F \equiv 0$.

    Let the point $B$ be any point of type $T$. By Remark \ref{rmk:hjBFixedPoint}, for $j \in \{1,2,3\}$, we have $h_{j,B} \geq 1$ and 
    $$h_{4,B}\geq 3\cdot (4-8-2)+21+1=4, \quad h_{5,B}\geq 3\cdot (5-8-2)+21+1=7.$$
    If $j\in \{ 6,7,8\}$, we observe that all six lines passing through $B$ are components of elements in $\mathfrak{L}(8;jB+Z)$. Therefore,
    $$h_{6,B}\geq \vdim(2;0B+Z'')=4, \quad h_{7,B}\geq \vdim(2;1B+Z'')=3, \quad h_{8,B}\geq \vdim(2;2B+Z'')=1,$$
    where again $Z''=T\setminus \{B\}.$ Thus $h_B=\sum h_{j,B}\geq 22$. From Theorem \ref{thm:main}, we have that $\deg F = 60$, and therefore $B$ is a fourfold point of $S$. This implies that the sextic $S$ consists of three lines through the points of $T$ taken twice.

    Take now the point $P$ from $F_3$ and a curve of degree $8$, consisting of lines $l_1,l_2,l_3$ and a quintic $Q$ through $Z_1$ and singular in $P$. Such $Q$ exists, as there exists a $4-$dimensional family of quintics through $Z_1$. This octic has a quintuple point in $P$, so $P$ is an element of the locus, a contradiction, as $P$ lies neither on lines $l_1, l_2, l_3$ nor on $S$. Thus, $F\equiv 0$.
\end{proof}  
\end{example}
    
\section{In $\PP^3$} 

Below we show how our method (i.e., investigating the locus of the interpolation matrix) works in $\PP^3$.

\begin{example}\label{ex:D4}

Take the $D4$ configuration of points in $\PP^3$. This configuration consists of $12$ points lying on $12$ planes, with each plane containing $6$ points and each point belonging to $6$ planes (see, e.g., \cite{HMN}). For this configuration, an unexpected hypersurface exists; it is a cone $(3; 3B+D4)$ (see \cite{HMNT, HMN}).  In our proof of the existence of the unexpected cone we again use only combinatorial data. 

To ensure that the interpolation matrix is square, we remove (any) two points from the configuration. Denote the remaining set of points $Z'$. The zero-locus of the determinant of the interpolation matrix, denoted by $\{F=0\}$, has degree $10$ (or is identically zero).  

Let $n_i$ denote the number of points on each of the $12$ planes, where $n_i$ is counted separately for the $i$-th plane. Clearly, before removing two points, we have $n_i=6$. After this operation, each $n_i$ can be either $4$, $5$, or $6$. Observe that each of the $12$ planes belongs to the locus $\{F=0\}$ with multiplicity $n_i-4$. Indeed, for a generic $B$ on a plane, $$h_{3,B} \geq \dim(3; 3B+Z')-\vdim(3; 3B+Z') = \dim(2;2B+(10-n_i) \text{ points} )=n_i-4.$$ Moreover, after removing two points, the sum of all $n_i$ is $60$. This can be seen as follows: initially, the sum of all $n_i$ is $6\cdot 12=72$. Since each removed point lies on $6$ planes, the total sum decreases by $2\cdot 6 = 12$, yielding a final sum of $60$.  

Summarizing, the total multiplicity of the planes in the locus is given by  
$$
\sum_{i=1}^{12}(n_i-4)=60-48=12.
$$ 
Since the locus $\{F=0\}$ has degree $10$ (or is identically zero), we conclude that $F\equiv 0$. Thus, the hypersurface $(3,3B+D4)$ exists.  

\begin{corollary}  
There exists a four-dimensional family of hypersurfaces $(4,4B+D4).$  
\end{corollary}  

Indeed, from the previous considerations, we immediately obtain a three-dimensional family, consisting of $(3,3B+D4)$ together with a plane passing through $B$. To obtain a four-dimensional family, we need one more quartic outside this three-dimensional space. The existence of such a quartic follows from the fact that the $12$ points of $D4$ lie on $4$ lines. Consequently, a cone spanned by $B$ and these lines forms the desired quartic. 

\end{example}

\begin{example}\label{ex:halfPenrose}

The Penrose configuration is described, for example, in  \cite{geproci, geproci1}. It consists of 40 points, given by the ideal $I = (xyzw, w(x^3 -y^3 + z^3), z(x^3 + y^3 + w^3), y(-x^3 + z^3 + w^3), x(y^3 + z^3 -w^3))$. We consider here a subconfiguration $Z$, consisting of 20 points (with coordinates written explicitly in \cite{geproci}), lying on 20 planes, 8 points on each plane. There are lines
defined by these planes, namely lines of type $l_j$, where $j$ planes intersect, for $j=2,3,4$. We want to prove the existence of (an unexpected) cone $(4, 4B+Z)$. We will only sketch the idea of the proof, omitting the parts where the combinatorics of $Z$ must be meticulously investigated.

Let $\cal P$ be a plane chosen among the 20 planes. Let $Z'=Z\setminus \cal P$ (12 points) and let
 $Z''=Z'\setminus \{\textrm{any two points of }Z'\}$. Observe that 
the degree of the locus of $(3, 3B+Z'')$ is 10. 

\begin{claim}\label{cla:prostewlocusie}

Let $B\in \cal P$. For $j=2,3,4$, all the lines $l_j$ on $\cal P$ are in the locus of $(3, 3B+Z'')$.
\end{claim}

The proof goes by considering lines according to their type.

For $l$ of type $l_2$, let $\cal P$ and $\cal P'$ be two planes defining $l$. Take $B\in l$.
There is a set $S$ of 6 points of $Z'$ outside $\cal P'$. We are interested in  $(2;2B+S).$
 The degree of the locus of $(2;2B+S)$ (i.e., of the determinant $F$ of the interpolation matrix)  is 4. The combinatorics of $Z$ implies that there exist five planes,
with four points out of these 6 on each. The 5 planes intersect $l$ in 5 different points. Thus, any such plane $\pi$ together with another plane (defined by the remaining two points out of the 6 points and by the intersection point $\pi\cap l$) gives an element of $(2; 2B+S).$ Thus, the locus $\{F=0\}$ of degree 4 intersects
$l$ in 5 points, so $l\subset \{F=0\}$.

Take now a point $B\in l$ and
any 10-point subset $Z''\subset Z'$. Observe that $(2;2B+S)+\cal P'$ is a member of ${\cal L}(3; 3B+Z')$, so also a member of ${\cal L}(3; 3B+Z'')$.

For $l$ of type $l_3$ take $B\in l\subset \cal P$. As $l$ is in two other planes and (combinatorics) the two planes contain 10 points of
$Z'$,  we may take a plane through the two remaining points and $B$, obtaining a member of ${\cal L}(3; 3B+Z')$ with $B\in l$. In consequence, $l$ is in the locus of $(3; 3B, Z'')$.

For $l$ of type $l_4$ the claim is easy. Take $B\in l\subset \cal P$. As $l$ is in three other planes, and all points of $Z'$ lie on these planes,
 we get a member of ${\cal L}(3; 3B+Z')$ with $B\in l$, and thus $l$ is in the locus of $(3; 3B, Z'')$.

Thus, all $l_2, l_3, l_4$ are in the locus of 
$(3; 3B+Z'')$ on $\cal P.$

We also claim that the following is true:
\begin{claim}\label{cla:Pwlocusie}
    Let $B\in \cal P$. Then $\cal P$ is contained in the locus of $(3; 3B+Z'')$.
\end{claim}
This follows from the fact that all the lines of types $2,3,4$ are in the locus of $(3, 3B+Z'')$.  There are  $12$ such lines on $\cal P$ ( the combinatorics is such that there are two 2 lines of type $ l_4$,  3 lines of type $l_3$ and 7 lines of type $l_2$) so $\cal P$ is in the locus of $(3, 3B+Z'')$ for \textit{any} choice of $Z''$, out of $Z'$. Thus, $\cal P$ is in the locus of $(3, 3B+Z')$.

Having  Claim \ref{cla:Pwlocusie} it is not difficult to see that:

\begin{corollary}
\begin{enumerate}
    \item   For $B\in \cal P$ there exists $(4, 4B+Z)$.
    \item For any 15 points from $Z$ and $F=\det M$, where $M$ is the interpolation matrix for $(4, 4B+Z)$, all 20 planes are in the set $\{F=0\}$.
\end{enumerate}  
\end{corollary}

Thus, we finally have

\begin{corollary}
    The number of planes in the locus $\{F=0\},$ counted with multiplicities is at least 21. Thus, $F\equiv 0.$
\end{corollary}

The proof of this corollary consists in analyzing the ways in which we can remove five points out of 20 points of $Z$.
There does not exist a subset of 15 points of $Z$ such that on each plane there are exactly six points. 
From this, it follows that there must exist a plane with 7 or more points; so it is in the locus at least twice.

\end{example}

\subsection*{Acknowledgements.}

Grzegorz Malara was supported by the National Science Centre (Poland), Sonata Grant \\ \textbf{2023/51/D/ST1/00118}.

\vspace{1cm}
\noindent
\textbf{Data Availability:} Data sharing not applicable to this article, as no datasets were generated or analyzed during
the current study.



\footnotesize
	\noindent
	Marcin Dumnicki, Halszka Tutaj-Gasi\'nska: Faculty of Mathematics and Computer Science, Jagiellonian University, Stanis{\l}awa {\L}ojasiewicza 6, 30-348 Kraków, Poland,\\
    \textit{E-mail address:} \texttt{marcin.dumnicki@uj.edu.pl, }\\
    \textit{E-mail address:} \texttt{halszka.tutajgasinska@gmail.com}\\

    Grzegorz Malara: Department of Mathematics, University of the National Education Commission, 	Podchor\c a\.zych 2, 30-084 Krak\'ow, Poland,\\
    \textit{E-mail address:}
    \texttt{grzegorzmalara@gmail.com}\\

\end{document}